\pdfoutput=1

\documentclass[10pt]{article}

\makeatletter
\renewcommand{\maketitle}{\bgroup\setlength{\parindent}{0pt}
\begin{flushleft}
  \textbf{\@title}

  \@author
\end{flushleft}\egroup
} 
\def\footnoterule{\kern-3\p@
  \hrule \@width 2in \kern 2.6\p@} 
\makeatother

\usepackage{preamble}
\usepackage{courier}
\usepackage[square,numbers]{natbib}

\title{{\Large \textbf{An arithmetic algebraic regularity lemma}}}
\author{Anand Pillay,\footnote[1]{Partially supported by NSF grant DMS-2054271.} Atticus Pascale Stonestrom\\{\small Department of Mathematics, University of Notre Dame}}

\begin{document}
\maketitle

{\small\noindent\textbf{Abstract:} We give an `arithmetic regularity lemma' for groups definable in finite fields, analogous to Tao's `algebraic regularity lemma' for graphs definable in finite fields. More specifically, we show that, for any $M>0$, any finite field $\mathbf{F}$, and any definable group $(G,\cdot)$ in $\mathbf{F}$ and definable subset $D\subseteq G$, each of complexity at most $M$, there is a normal definable subgroup $H\leqslant G$, of index and complexity $O_M(1)$, such that the following holds: for any cosets $V,W$ of $H$, the bipartite graph $(V,W,xy^{-1}\in D)$ is $O_M(|\mathbf{F}|^{-1/2})$-quasirandom. Various analogous regularity conditions follow; for example, for any $g\in G$, the Fourier coefficient $||\widehat{1}_{H\cap Dg}(\pi)||_{\mathrm{op}}$ is $O_M(|\mathbf{F}|^{-1/8})$ for every non-trivial irreducible representation $\pi$ of $H$. \newline 

\noindent\textbf{Notation:} Definable means `definable with parameters', and definable of complexity at most $M$ means definable by a formula of length at most $M$. A definable group of complexity at most $M$ is a definable group such that the formulas defining the underlying set of the group and the multiplication operation each have complexity at most $M$. Given a finite set $X$ and $f:X\to\mathbb{C}$, we have the $L^2$ and $L^\infty$ norms $||f||_2^2=\sum_{x\in X}|f(x)|^2$ and $||f||_\infty=\max_{x\in X}|f(x)|$. For a bounded linear map $f:V\to W$ between normed vector spaces, the operator norm $||f||_{\mathrm{op}}$ is $\sup_{v\in V\setminus 0}||f(v)||/||v||$. By a bipartite graph, we mean a triple $(V,W,E)$, where $E\subseteq V\times W$, and we define $E(v,W)=\{w\in W:(v,w)\in E\}$ and define $E(V,w)$ symmetrically, for all $v\in V$ and $w\in W$. If unambiguous, we will also write $N_v$ for $v\in V$ to mean $E(v,W)$ and $N_w$ for $w\in W$ to mean $E(V,w)$. If $\mathfrak{C}$ is a structure, $A\subset\mathfrak{C}$ is a subset, and $b,c$ are tuples from $\mathfrak{C}$, we write $b\equiv_A c$ to mean $\tp(b/A)=\tp(c/A)$. Given a set of parameters $A\subseteq\mathfrak{C}$ and an $A$-definable group $G$, we use $S_G(A)$ to denote the set of complete types over $A$ containing the formula defining $G$.

\section{Introduction}
Szemeredédi's regularity lemma applies to arbitrary graphs, and there is a vast literature on the strengthenings one can obtain by placing some additional hypotheses on the graphs one considers. One example of this kind is to look at bipartite graphs of the form (for example) $(G,G,xy^{-1}\in D)$, where $(G,\cdot)$ is a group and $D\subseteq G$ is a subset. The desired sort of result for such a graph is then to obtain a Szemerédi-style partition of the two copies of $G$, but where the partition is into algebraically well-structured sets: for example, cosets of a subgroup or translates of a Bohr set. With appropriate hypotheses on $G$ and $D$, this program has been carried out successfully in a number of different contexts; see for example \cite{chung_graham}, \cite{green}, \cite{gowers_groups}, \cite{terry_wolf_1}, \cite{terry_wolf_2}, \cite{alon_fox_zhao}, \cite{conant_pillay_terry_1}, \cite{conant_pillay_terry_2}, and \cite{conant}.

In a different direction, another variation for Szemerédi's regularity lemma was proved by Tao in \cite{tao} for families of graphs uniformly definable in finite fields, improving upon Szemerédi regularity both in the absence of irregular pairs and in the `power-saving' bound on the degree of regularity. Several alternative proofs were subsequently given, independently by Pillay and Starchenko in \cite{pillay_starchenko} and by Hrushovski in private communication, and then subsequently by Tao in \cite{tao_blog}. More recently, a hypergraph version of this theorem was proved by Chevalier and Levi in \cite{chevalier_levi}.

In this paper we merge the two themes above, proving an `arithmetic' version of Tao's theorem. In particular we prove the following, which are Theorem \ref{main_theorem} and Corollary \ref{main_theorem_fourier} respectively. (See Section \ref{qr_prelim_section} for definitions.) One typical example to keep in mind for the following result is that of the bipartite `Paley graph', where the definable group is the additive group and the distinguished subset is the set of quadratic residues. These graphs are well-known to be quasirandom; see for example \cite{zhao}.

\begin{theorem}
    For any $M$, there is a positive constant $C>0$ such that the following holds. Suppose $\mathbf{F}$ is a finite field, and that $G$ is a definable group in $\mathbf{F}$ and $D\subseteq G$ is a definable subset, both of complexity at most $M$. Then there is a definable normal subgroup $H\leqslant G$, of index and complexity at most $C$, such that, for any cosets $V,W$ of $H$, the bipartite graph $(V,W,xy^{-1}\in D)$ is $C|\mathbf{F}|^{-1/2}$-quasirandom.
\end{theorem}
\begin{corollary}
    For any $M$, there is a positive constant $C>0$ such that the following holds. Suppose $\mathbf{F}$ is a finite field, and that $G$ is a definable group in $\mathbf{F}$ and $D\subseteq G$ is a definable subset, both of complexity at most $M$. Then there is a definable normal subgroup $H\leqslant G$, of index and complexity at most $C$, such that, for any $g\in G$, we have $||\widehat{1}_{H\cap Dg}(\pi)||_{\mathrm{op}}\leqslant C|\mathbf{F}|^{-1/8}$, for every non-trivial irreducible representation $\pi$ of $H$, where we use the convention $\widehat{f}(\pi)=\frac{1}{|H|}\sum_{h\in H}f(h)\pi(h^{-1})$ for the Fourier transform on $H$ of a function $f:H\to\mathbb{C}$.
\end{corollary} In Section \ref{special_cases_sec} we make some comparisons with results from \cite{green} and \cite{gowers_groups}, and discuss a few special cases of our result in which one can avoid passing to a proper subgroup $H$; this occurs most generally if $G$ is the $\mathbf{F}$-points of a simply connected algebraic group defined over $\mathbf{F}$, or if the characteristic of $\mathbf{F}$ is sufficiently large and $G$ is the $\mathbf{F}$-points of a simply connected algebraic group over $\mathbb{Z}$ that is simply connected over any field of characteristic $0$. The latter case applies for example to the additive group, and in that case our result relates in a similar way to the `finite field model' theorem from \cite{green} as Tao's result from \cite{tao} relates to Szemerédi's regularity lemma; we remark on these connections in Section \ref{additive group section}. On the other hand, the former case applies in particular to \textit{semisimple} simply connected algebraic groups; in this case, however, vastly more can be said, and \textit{every} subset is quasirandom, not just definable ones. This is well-known to experts, and already remarked upon in a few places – for example in Breuillard's notes \cite{breuillard_2} – but for completeness we give a complete discussion of it in Section \ref{semisimple case section}.

We also briefly remark that there are two other more general contexts to which we expect the proofs here to generalize; one is to the setting of Section 6.2 in \cite{garcia_macpherson_steinhorn}, of pseudofinite structures satisfying certain additional conditions. The other is to the setting of \cite{chevalier_levi}, of definable sets of finite total dimension uniformly definable in the difference fields $(\mathbf{F}_q^{\mathrm{alg}},x\mapsto x^{q})$. However we do not pursue these connections explicitly here.

\section{Preliminaries}
\subsection {Quasirandomness}\label{qr_prelim_section}
Quasirandom graphs were first defined in \cite{chung_graham_wilson}. We will need the bipartite version of the notion, introduced in \cite{gowers_graphs}.
\begin{definition}\label{qr_defn}
    A finite bipartite graph $(V,W,E)$ with $|E|=\delta|V||W|$ is $\varepsilon$-quasirandom if $$\sum_{v,v'\in V}\sum_{w,w'\in W}1_E(v,w)1_E(v,w')1_E(v',w)1_E(v',w')\leqslant(\delta^4+\varepsilon)|V|^2|W|^2;$$ equivalently, if we have $\sum_{v,v'\in V}|E(v,W)\cap E(v',W)|^2\leqslant(\delta^4+\varepsilon)|V|^2|W|^2$, or equivalently if we have the analogous symmetric bound with the roles of $V$ and $W$ reversed.
\end{definition}(Note that, by two applications of Cauchy-Schwarz, the left hand side in the quantity above is always bounded below by $\delta^4|V|^2|W|^2$.) Recall also the following closely related notions:

\begin{definition}
    Let $(V,W,E)$ be a finite bipartite graph with $|E|=\delta|V||W|$. Then:
    \begin{enumerate}
        \item $(V,W,E)$ is $\varepsilon$-regular if, for all $A\subseteq V$ and $B\subseteq W$ with $|A|\geqslant\varepsilon|V|$ and $|B|\geqslant\varepsilon|W|$, we have $\left||E\cap (A\times B)|-\delta|A||B|\right|\leqslant\varepsilon|A||B|$.
        \item $(V,W,E)$ is weakly $\varepsilon$-regular if, for all $A\subseteq V$ and $B\subseteq W$, we have $\big||E\cap (A\times B)-\delta|A||B|\big|\leqslant\varepsilon|V||W|$.
    \end{enumerate} Note that $\varepsilon$-regularity implies weak $\varepsilon$-regularity, and weak $\varepsilon$-regularity implies $\varepsilon^{1/3}$-regularity.
\end{definition}

(Condition 2 in the above definition is the one that appears more naturally in the setting of \cite{gowers_graphs} and in the proof of our theorem. In \cite{tao}, Tao refers to it as `$\varepsilon$-regularity,' but this conflicts with the usual terminology from the Szemerédi regularity theorem, which is why we call it `weak $\varepsilon$-regularity'.) Now the following is due to Gowers, from \cite{gowers_graphs} and \cite{gowers_groups}.

\begin{fact}\label{qr_thm}
    Let $(V,W,E)$ be a bipartite graph with $|E|=\delta|V||W|$, and let $M$ be its $|W|\times|V|$-adjacency matrix. Then the following are polynomially equivalent:
    \begin{enumerate}
        \item $(V,W,E)$ is $\varepsilon_1$-quasirandom.
        \item $(V,W,E)$ is weakly $\varepsilon_2$-regular.
        \item For all $f\in\mathbb{C}^V$ with $\sum_{v\in V}f(v)=0$, we have $||Mf||_2/||f||_2\leqslant\varepsilon_3|V|^{1/2}|W|^{1/2}$.
    \end{enumerate} More precisely, if the first condition holds, then the second and third condition hold with $\varepsilon_2=\varepsilon_1^{1/4}$ and $\varepsilon_3=\varepsilon_1^{1/4}$. Conversely, if the second condition holds, then the first holds with $\varepsilon_1=12\varepsilon_2$, and if the third condition holds, then the first holds with $\varepsilon_1=\delta\varepsilon_3^2$.
\end{fact}

\subsection{Quasirandom subsets of groups}\label{quasirandom subsets of groups section}
Given a finite group $H$ and a subset $D\subseteq H$, quasirandomness of the bipartite graph $(H,H,xy^{-1}\in D)$ is (polynomially) equivalent to a bound on the Fourier coefficients of the indicator function $1_D:H\to \mathbb{C}$. This connection is widely used in additive combinatorics in abelian groups, and was perhaps first observed in \cite{chung_graham}. Although it is surely also well-known in the non-abelian case, and for instance is used implicitly throughout \cite{gowers_groups}, we could not find a reference making the connection explicit, so we include the details for completeness. See Section 8 of \cite{breuillard} for a very nice treatment of the material, which does not explicitly state the observation below but clearly has it in mind.

Let us first recall some basic facts about non-abelian Fourier analysis. We describe only the situation for finite groups, but by the Peter-Weyl theorem the machinery of non-abelian Fourier analysis all works working over a compact group equipped with Haar measure.

Let $H$ be a finite group, and let $\widehat{H}$ denote the set of irreducible complex representations of $H$; unlike in the abelian case, $\widehat{H}$ does not in general have the structure of a group. Given a function $f:H\to\mathbb{C}$, we define the Fourier transform $\widehat{f}$ which takes an element $\pi:H\to\mathrm{GL}(V_\pi)$ of $\widehat{H}$ to the operator $\frac{1}{|H|}\sum_{g\in H}f(g)\pi(g^{-1})\in\mathrm{End}(V_\pi)$. The basic properties of the Fourier transform in the abelian case, namely the Fourier inversion formula, the Parseval-Plancherel identity, and the identity for the Fourier transform of a convolution, all have appropriate analogues in the non-abelian case. Given a subset $D\subseteq H$, following \cite{breuillard}, let us say that $D$ is $\varepsilon$-quasirandom if $\max_{\pi\in\widehat{H},\pi\neq\pi_{\mathrm{triv}}}||\widehat{1}_{D^{-1}}(\pi)||_{\mathrm{op}}\leqslant\varepsilon$, ie if $\max_{\pi\in\widehat{H},\pi\neq\pi_{\mathrm{triv}}}||\sum_{d\in D}\pi(d)||_{\mathrm{op}}\leqslant\varepsilon|H|$.

\begin{lemma}\label{qr_subsets_of_groups_1}
    $D\subseteq H$ is $\varepsilon$-quasirandom iff we have $||Mf||_2/||f||_2\leqslant\varepsilon|H|$ for all $f\in\mathbb{C}^H$ with $\sum_{g\in H}f(g)=0$, where $M$ is the adjacency matrix of the graph $(H,H,xy^{-1}\in D)$.
\end{lemma}
\begin{proof}

Consider the function space $L^2(H)=\mathbb{C}^H$ equipped with the inner product $\langle f,g\rangle=\sum_{x\in H}\overline{f(x)}g(x)$. The usual left action $(\tau(x)f)(y)=f(x^{-1}y)$ of $H$ on $\mathbb{C}^H$ gives a unitary representation $\tau:H\to\mathrm{GL}(\mathbb{C}^H)$, which is isomorphic to the direct sum $\bigoplus_{\pi\in \widehat{H}}d_\pi V_\pi$ of irreducible representations of $H$ with multiplicities $d_\pi=\dim(V_\pi)$. On the other hand, we have $\mathbb{C}^H=\mathbb{C}1_H\oplus W$, where $1_H$ is the constant function $x\mapsto 1$ and $$W=\{f\in\mathbb{C}^H:\sum_{x\in H}f(x)=0\},$$ and $\mathbb{C}1_H$ and $W$ are $H$-invariant. So it follows that $W\cong\bigoplus_{\pi\in\widehat{H},\pi\neq\pi_{\mathrm{triv}}}d_\pi V_\pi$.

Now, consider the endomorphism $\phi_D=\sum_{d\in D}\tau(d)$ of $\mathbb{C}^H$. The operator norm of a direct sum of endomorphisms is the maximum of their respective operator norms, so we have $||\phi_D\mathord{\upharpoonright}_W||_{\mathrm{op}}=\max_{\pi\in\widehat{H},\pi\neq\pi_{\mathrm{triv}}}||\phi_D\mathord{\upharpoonright}_{V_\pi}||_{\mathrm{op}}$. So $D$ is $\varepsilon$-quasirandom if and only if $||\phi_D\mathord{\upharpoonright}_W||_{\mathrm{op}}\leqslant\varepsilon |H|$.

For $g\in H$, let $\iota(g)\in\mathbb{C}^H$ be the indicator function of $\{g\}$. Thus $\{\iota(g):g\in H\}$ is a basis for $\mathbb{C}^H$ and for $x,y\in H$ we have $\tau(x)\iota(y)=\iota(xy)$. So $\phi_D(\iota(y))=\sum_{d\in D}\iota(dy)$. Note that $\iota(x)$ has coefficient $1$ in this sum if and only if $xy^{-1}\in D$. So the matrix representation of $\phi_D$ with respect to (any enumeration of) the basis $\{\iota(g):g\in H\}$ is the same as the adjacency matrix $M$ (with respect to the corresponding enumeration of $H$). So $||\phi_D\mathord{\upharpoonright}_W||_{\mathrm{op}}=\sup_{f\in W}||Mf||_2/||f||_2$ and the claim follows.
\end{proof}

Now from Fact \ref{qr_thm} we have the following consequence.

\begin{corollary}\label{qr_subsets_of_groups_2}
    Let $D\subseteq H$ and consider the following two conditions.
    \begin{enumerate}
        \item $D$ is $\varepsilon$-quasirandom.
        \item $(H,H,xy^{-1}\in D)$ is $\varepsilon'$-quasirandom.
    \end{enumerate} If the second condition holds, then the first condition holds with $\varepsilon=(\varepsilon')^{1/4}$. If the first condition holds, then the second condition holds with $\varepsilon'=\varepsilon^2$.
\end{corollary}

Now we can state some theorems from \cite{gowers_groups} and \cite{green} using the language of Lemma \ref{qr_subsets_of_groups_1}. From \cite{gowers_groups} we have the following (see \cite{breuillard} too for a more succinct proof):

\begin{fact}\label{qr_dimension}
    Let $d$ be the minimal dimension of a non-trivial irreducible representation of $H$. Then every subset of $H$ is $d^{-1/2}$-quasirandom. Conversely, if every subset of $H$ is $\varepsilon$-quasirandom, then any non-trivial irreducible representation of $H$ has dimension at least $\varepsilon^{-2/3}/100$.
\end{fact}
\begin{example}\label{sl2_example}
    If $q$ is a prime power, the group $H=\mathrm{SL}_2(\mathbf{F}_q)$ has no non-trivial irreducible representation of dimension $<(q-1)/2$, and thus every subset of it is $2q^{-1/2}$-quasirandom. Hence by Corollary \ref{qr_subsets_of_groups_2}, the graph $(H,H,xy^{-1}\in D)$ is $4q^{-1}$-quasirandom for every $D\subseteq H$.
\end{example}

On the opposite end of the spectrum, the paper \cite{green} proves a regularity theorem for arbitrary subsets of \textit{abelian} groups. The theorem statement is substantially simplified working in a `finite field' model, ie in a group of form $(\mathbb{Z}/p\mathbb{Z},+)^n$, where $p$ is a fixed prime number and $n$ is large. This case is treated in Section 2 of \cite{green}, where the following is proved; in Section 2 only the case $p=2$ is stated explicitly, but the same argument works for an arbitrary but fixed $p$. 

\begin{fact}\label{green_theorem}
    Fix any prime $p$ and $\varepsilon\in(0,1/2)$. Then there is $C>0$ such that the following holds. For any $n\in\mathbb{N}$, if $G=(\mathbf{F}_{p^n},+)$ and $D\subseteq G$ is an arbitrary subset, then there is a subgroup $H\leqslant G$ of index at most $C$ and such that, for all but $\varepsilon|G|$-many $g\in G$, the intersection $(g+D)\cap H$ is an $\varepsilon$-quasirandom subset of $H$.
\end{fact}

\subsection{Simple theories}
Let $T$ be a complete $L$-theory and let $\mathfrak{C}$ be a saturated model. Given an $L$-formula $\phi(x,y)$, a tuple $b\in\mathfrak{C}^{y}$, and a parameter set $A\subset\mathfrak{C}$, we say that a formula $\phi(x,b)$ divides over $A$ if there is some $k\in\omega$ and some sequence $(b_i)_{i\in\omega}\in\mathfrak{C}^y$ such that $b_i\equiv_A b$ for all $i\in\omega$ and such that, for any $i_1<\dots<i_k\in\omega$, the conjunction $\phi(x,b_{i_1})\wedge\dots\wedge\phi(x,b_{i_k})$ is inconsistent. We say that a partial type forks over $A$ if it implies a disjunction of finitely many formulas that each divide over $A$. Finally, given a set $A$ and sets or tuples $b,c$, we say that $b$ is non-forking-independent from $c$ over $A$, written $b\ind_A c$, if $\tp(b/A,c)$ does not fork over $A$.

`Simplicity' is a property of a first-order theory introduced in \cite{shelah_1} and \cite{shelah_2}. It was proved in \cite{kim} and \cite{kim_pillay} that the relation $\ind$ behaves well in simple theories, and that simplicity of a theory can be characterized by the good behavior of $\ind$. Rather than give the original definition of simplicity, we will instead give the above-mentioned characterization of it; for a thorough introduction to simple theories in general, see \cite{wagner}.

\begin{fact}\label{kim_pillay}
    If $T$ is simple, then the relation $\ind$ satisfies the following properties:
    \begin{enumerate}
        \item (Finite character:) $a\ind_C B$ iff $a\ind_C B_0$ for every finite subset $B_0$ of $B$.
        \item (Existence:) For all $a,b,C$, there is $a'\equiv_C a$ such that $a'\ind_Cb$.
        \item (Monotonicity and transitivity:) $b\ind_A (c,d)$ iff $b\ind_{A}c$ and $b\ind_{(A,c)}d$.
        \item (Symmetry:) $a\ind_C b$ iff $b\ind_C a$.
        \item (Local character:) There is a cardinal $\kappa$ such that, for all $a,B$, there is $B_0\subseteq B$ of size at most $\kappa$ such that $a\ind_{B_0}B$.
        \item (Independence theorem over models:) Let $M$ be a model, and suppose that $a,a',b,c$ are tuples such that $a\equiv_M a'$ and $a\ind_M b$ and $a'\ind_M c$ and $b\ind_M c$. Then there is $a''$ such that $a''\equiv_{(M,b)}a$ and $a''\equiv_{(M,c)}a'$ and $a''\ind_M(b,c)$.
    \end{enumerate}Conversely, if $\ind$ satisfies either 4 or 5, then $T$ is simple. Moreover, if there is any automorphism-invariant ternary relation $\ind'$ on $\mathfrak{C}$ satisfying all of the above properties, then $T$ is simple and $\ind'=\ind$.
\end{fact} When working in a model of a simple theory, we will just say `the independence theorem' to refer to the sixth property above.

\subsection{Groups in simple theories}\label{groups_simple_thys_sec}
We will use many of the fundamental properties of groups definable in simple theories, as developed in \cite{pillay}. Suppose that $T$ is a simple theory, that $\mathfrak{C}$ is a saturated model of $T$ and that $G$ is a group definable in $\mathfrak{C}$, possibly with parameters. Recall that, for a parameter set $B\subset\mathfrak{C}$, $G^{00}_B$ denotes the smallest `bounded-index' subgroup of $G$ type-definable over $B$, where bounded-index means of index smaller than the degree of saturation of $\mathfrak{C}$.

\begin{definition}
    Let $\phi(x)$ be a formula with parameters concentrated on $G$. Then $\phi(x)$ is (left) f-generic if, for every small set $B$ over which $\phi(x)$ and $G$ are both defined, the formula $g\phi(x)$ does not fork over $B$ for all $g\in G(\mathfrak{C})$.
\end{definition}

As usual, we will say that a partial type is f-generic if it does not imply any non-f-generic formula. The properties of f-generic types for groups in simple theories were developed in \cite{pillay},\footnote[1]{In \cite{pillay}, the notion with the properties listed below is called genericity, but it was renamed to f-genericity in \cite{newelski_petrykowski}. The original definition given in \cite{pillay} is condition 5 below.} some of which we summarize below.

\begin{fact}\label{f-gen_facts}
\begin{enumerate}
    \item Right f-genericity and left f-genericity coincide, so in particular a partial type $\pi(x)$ is f-generic iff $\pi(x)^{-1}$ is.
    \item The non-f-generic formulas form a (proper) ideal, so an f-generic partial type can be extended to an f-generic complete type over any parameter set.
    \item Let $A$ be a small set of parameters over which $G$ is defined. Then a partial type is left f-generic iff no left translate of it forks over $A$.
    \item A complete type is f-generic iff some (all) of its non-forking extensions are.
    \item Fix a small set of parameters $A$ over which $G$ is defined, and let $B\supseteq A$. If $\tp(c/B)$ is f-generic, and $g\in G(\mathfrak{C})$ with $c\ind_B g$, then $gc\ind_A (B,g)$. (This follows from points 3 and 4.)
    \item Fix a small set of parameters $C$ over which $G$ is defined. If $\tp(a/C)$ and $\tp(b/C)$ are f-generic, and $a\ind_C b$, then $(a^{-1}b,a,b)$ is a pairwise $C$-independent triple, each element of which is f-generic over $C$. (This follows from points 1, 3, and 4.)
    \item Let $M$ be a small model over which $G$ is definable, and suppose $q\in S_G(M)$ is f-generic. If $g\in G^{00}_M$ and $\tp(g/M)$ is f-generic, then there exists some $b\models q$ with $b\ind_M g$ and $gb\models q$. (In the notation of \cite{pillay}, this says $g\in S(q)$.)
\end{enumerate}
\end{fact}

Now the following is proved in Proposition 2.2 in \cite{pillay_scanlon_wagner}; it is only stated there in the case $\mathfrak{c}=\mathfrak{d}=G^{00}_M$, but the same proof works in general – see for instance \cite{martin-pizarro_pillay}. Note that, since $G^{00}_M$ is bounded-index and type-definable over $M$, every coset of $G^{00}_M$ is also type-definable over $M$.\footnote{More generally, if $\mathfrak{C}$ is a saturated model of any theory and $E(x,y)$ is a bounded equivalence relation on $\mathfrak{C}^x$ type-definable over some small model $M\prec\mathfrak{C}$, then every coset of $E$ is type-definable over $M$.} Thus in particular every type in $S_G(M)$ is concentrated on some coset of $G^{00}_M$ and every coset of $G^{00}_M$ is concentrated on by some type in $S_G(M)$.

\begin{fact}\label{pq=r}
Let $p,q,r\in S_G(M)$ be f-generic types of $G$ such that, if $\mathfrak{c},\mathfrak{d}$ are the respective cosets of $G^{00}_M$ on which $p,q$ are concentrated, then $r$ is concentrated on $\mathfrak{c}\mathfrak{d}$. Then there are $a\models p$ and $b\models q$ such that $ab\models r$ and such that $(a,b,ab)$ is pairwise $M$-independent.
\end{fact}

\subsection{S1 theories}\label{s1_section}
We will need in this paper the compatibility of the results from Section \ref{groups_simple_thys_sec} with the behaviour of dimension (in the sense of algebraic closure) in pseudofinite fields. The key things we need are that (i) the theory of pseudofinite fields is simple, (ii) dimension-independence coincides with nonforking independence, and (iii) if $G$ is a definable group, then the types in $G$ of maximal dimension coincide with f-generic types. Points (i) and (ii) are well-known, and were proved in \cite{kim_pillay} using Fact \ref{kim_pillay} and results from \cite{hrushovski_pillay}.\footnote[1]{Note that \cite{hrushovski_pillay} was written before the general theory of simple theories was developed in \cite{kim} and \cite{kim_pillay}, and forking and dividing do not explicitly appear anywhere in \cite{hrushovski_pillay} other than in the context of local stability.} It was also mentioned in \cite{kim_pillay} that a more direct method of proving simplicity would be desirable. For completeness, we take the liberty here to give such a direct account of (i) and (ii), along with (iii). The proofs are completely routine. See also Corollary 3.9 of \cite{elwes_macpherson}.

Let $T$ be a complete one-sorted $L$-theory and let $\mathfrak{C}$ be a saturated model of $T$. Recall from Definition 2.1 of \cite{hrushovski_pillay} that $T$ is said to be `geometric' if (i) algebraic closure defines a pregeometry on $\mathfrak{C}$, and (ii) $T$ has `elimination of $\exists^\infty$': for any $L$-formula $\phi(x,y)$, there is $n_\phi$ such that, for any $b\in\mathfrak{C}^{|y|}$, $\phi(\mathfrak{C}^{|x|},b)$ is finite if and only if it has size at most $n_\phi$. For a definable set $X$, we let $\dim(X)$ denote the acl-dimension of $X$. If $T$ is geometric, $a$ is a finite tuple, and $C$ is a parameter set, we let $\dim(a/C)$ be the length of a maximal $\mathrm{acl}_C$-independent subtuple of $a$. For a $C$-definable set $X$ we let $\dim(X)=\max\{\dim(a/C):a\in X\}$; this is well-defined, and $\dim(a/C)=\min\{\dim(X):X\ni a,X\text{ is }C\text{-definable}\}$. Finally, given finite tuples $a,b$, we say that $a$ is dimension-independent from $b$ over $C$ if $\dim(a/C,b)=\dim(a/C)$, and we say that a sequence $a_1,\dots,a_n$ is dimension-independent over $C$ if $a_i$ is dimension-independent from $\{a_j:j\neq i\}$ over $C$.

Now the following notion is defined in \cite{hrushovski_pillay}:
\begin{definition}
    A geometric theory $T$ has property S1 if one cannot find a definable set $X$, a formula $\phi(x,y)$, and tuples $b_i\in\mathfrak{C}^{|y|}$, such that: (i) $\phi(x,b_i)\to x\in X$ for all $i\in\omega$, (ii) $\dim(\phi(x,b_i))=\dim(X)$ for all $i\in\omega$, and (iii) $\dim(\phi(x,b_i)\wedge\phi(x,b_j))<\dim(X)$ for all $i\neq j$.
\end{definition}

\begin{lemma}
    Suppose $T$ is a geometric theory with property S1. Then dimension-independence and dividing-independence coincide in $\mathfrak{C}$. In particular $T$ is supersimple of SU-rank $1$.
\end{lemma}
\begin{proof}
    Fix a finite tuple $a$ and parameter sets $C\subseteq B$. We will show that $\dim(a/B)<\dim(a/C)$ iff $\tp(a/B)$ divides over $C$. It suffices to show this in the case where $B=(C,b)$ for some finite tuple $b$. Let $r(x,y)=\tp(a,b/C)$.

    In one direction, suppose $\dim(a/C,b)<\dim(a/C)$. By exchange, also $\dim(b/C,a)<\dim(b/C)$. Let $n=\dim(a/C)$. We claim that, for any dimension-independent $b_0,\dots,b_n$ realizing $\tp(b/C)$, $\bigwedge_{i=0}^nr(x,b_i)$ is inconsistent. Otherwise let $a'$ realize it. Then 
    \begin{align*}\dim(b/C)(n+1)&=\dim(b_0,\dots,b_n/C)\\&\leqslant\dim(a',b_0,\dots,b_n/C)\\&= \dim(a'/C)+\dim(b_0,\dots,b_n/C,a')\\&\leqslant n+\sum_{i=0}^n\dim(b_i/C,a')\\&\leqslant n+\dim(b/C,a)(n+1),
    \end{align*}contradicting that $\dim(b/C,a)\leqslant\dim(b/C)-1$. So indeed $\bigwedge_{i=0}^nr(x,b_i)$ is inconsistent for any dimension-independent tuple $(b_0,\dots,b_n)$ of realizations of $\tp(b/C)$, and hence $r(x,b)$ divides over $C$.

    In the other direction, suppose that $r(x,b)$ divides over $C$, and fix an $L(C)$-formula $\phi(x,y)\in r(x,y)$ such that $\phi(x,b)$ divides over $C$. Let $n=\dim(a/C)$ and let $\psi(x)$ be an $L(C)$-formula with $a\models\psi(x)$ and $\dim(\psi(x))=n$. Replacing $\phi(x,y)$ with $\phi(x,y)\wedge\psi(x)$, we may assume that $\phi(x,b)\to\psi(x)$. We claim that $\dim(\phi(x,b))<n$, which will show that $\dim(a/C,b)<\dim(a/C)$ and give the desired result. Suppose otherwise that $\dim(\phi(x,b))=n$.

    Since $\phi(x,b)$ divides over $C$, there is a $C$-indiscernible sequence $(b_i)_{i\in\omega}$ with $b_0=b$ and such that $\{\phi(x,b_i):i\in\omega\}$ is $k$-inconsistent. In particular, the intersection of any $k$ of the $\phi(x,b_i)$ is empty, and since $\dim(\phi(x,b_i))=n$ for all $i$ there must be some $k'<k$ maximal such that $\phi(x,b_1)\wedge\dots\wedge\phi(x,b_{k'})$ has dimension $n$. But then if we let $c_i=(b_{ik'+1},\dots,b_{(i+1)k'})$ and let $\theta(x,c_i)=\phi(x,b_{ik'+1})\wedge\dots\wedge \phi(x,b_{(i+1)k')})$, then $\dim(\theta(x,c_i))=\dim(\psi(x))$ for all $i$, $\theta(x,c_i)\to\psi(x)$ for all $i$, and $\dim(\theta(x,c_i)\wedge\theta(x,c_j))<\dim(\psi(x))$ for all $i\neq j$, contradicting property S1.

    So $\tp(a/B)$ divides over $C\subseteq B$ iff $\dim(a/B)<\dim(a/C)$. It follows that every complete type does not divide over a finite subset of its domain, so $T$ is supersimple, and it has SU-rank $1$ since $\dim(x=x)=1$.
    \end{proof}

Now let us turn to definable groups in S1 theories. First we note the group version of the previous lemma:
\begin{lemma}\label{f-gen_iff_max_dim}
    Suppose $T$ is a geometric theory with property S1, and let $G$ be a $\varnothing$-definable group in $T$. Then a formula concentrating on $G$ is f-generic iff it has dimension $\dim(G)$.
\end{lemma}
\begin{proof}
    The proof is essentially identical to the previous lemma. Fix a definable set $X\subseteq G$, and let $C$ be a small parameter set over which $G,X$ are both defined. Since $T$ is simple by the previous lemma, we may use Fact \ref{f-gen_facts}.
    
    The proof of the previous lemma shows that, if $C$ is a parameter set, $\psi(x),\phi(x,y)$ are $L(C)$-formulas, and $\phi(x,b)$ is an instance of $\phi$ that divides over $C$ and implies $\psi(x)$, then $\dim(\phi(x,b))<\dim(\psi(x))$. On the other hand, if $X$ is not f-generic, then there is some $g\in G$ such that $x\in gX$ divides over $C$, which hence forces $\dim(gX)<\dim(G)$. But $\dim(X)=\dim(gX)$ so $\dim(X)<\dim(G)$ as needed.
    
    In the other direction, suppose $\dim(X)<\dim(G)$, and suppose for contradiction that $X$ is f-generic. Then there is $a\in X$ such that $\tp(a/C)$ is f-generic. Pick any $g\in G$ with $\dim(g/C,a)=\dim(G)$; by the previous lemma, $g\ind_C a$, so that $a\ind_C g$. By Fact \ref{f-gen_facts}(5), we have $ga\ind_C g$, so again by the previous lemma $\dim(ga/C)=\dim(ga/C,g)$. But $\tp(ga/C,g)$ contains the formula $x\in gX$, and $\dim(gX)=\dim(X)<\dim(G)$, so that $\dim(ga/C)<\dim(G)$. This contradicts that $\dim(g/C,a)=\dim(G)$.
\end{proof}

We will also use crucially the following fact, which is Lemma 6.1 in \cite{hrushovski_pillay}.

\begin{fact}\label{g0=g00}
    Suppose that $T$ is a geometric theory with property S1, that $M\prec\mathfrak{C}$ is a small model, and that $G$ is a $M$-definable group of $T$. Any subgroup of $G$ type-definable over $M$ is an intersection of $M$-definable subgroups of $G$. In particular, $G^{00}_M$ is an intersection of $M$-definable finite-index subgroups of $G$.
\end{fact}

\subsection{Pseudofinite fields}\label{psf_section}
Recall that a pseudofinite field is an infinite field elementarily equivalent (in the language of rings) to some ultraproduct of finite fields. In \cite{ax}, Ax gave a first-order axiomatization for the class of pseudofinite fields, which has since proved an extremely rich area of study in model theory. In light of the previous section, we first quote the following, which are noted in Proposition 2.11 and Proposition 2.18 of \cite{hrushovski_pillay} as consequences of the results of \cite{chatzidakis_vdd_macintyre}:

\begin{fact}
    The complete theory of any pseudofinite field is geometric and S1. In particular, all of the facts from Section \ref{s1_section} apply to pseudofinite fields.
\end{fact}

The following is the main result of \cite{chatzidakis_vdd_macintyre}. (Point 3 is not mentioned there, but that it can be included follows easily.)

\begin{fact}\label{psf}
    For any formula $\phi(x,y)$ in the language of rings, there is a constant $C\geqslant 0$, a finite set $D\subseteq\mathbb{N}\times\mathbb{Q}_{>0}$, and formulas $\theta_{(d,r)}(y)$ for each $(d,r)\in D$, such that the following hold for every finite field $\mathbf{F}$:
    \begin{enumerate}
        \item For all $b\in\mathbf{F}^{|y|}$, there is $(d,r)\in D$ with $\mathbf{F}\models\theta_{(d,r)}(b)$.
        \item If $\mathbf{F}\models\theta_{(d,r)}(b)$ and $\phi(\mathbf{F}^{|x|},b)\neq\varnothing$, then $\left||\phi(\mathbf{F}^{|x|},b)|-r|\mathbf{F}|^d\right|\leqslant C|\mathbf{F}|^{d-1/2}$.
        \item If $\mathbf{F}\models\theta_{(d,r)}(b)$ and $\phi(\mathbf{F}^{|x|},b)\neq\varnothing$, then $|\mathbf{F}|^d/C\leqslant|\phi(\mathbf{F}^{|x|},b)|\leqslant C|\mathbf{F}|^d$.
    \end{enumerate}
\end{fact}

In a finite or pseudofinite field $\mathbf{F}$, we say that a non-empty definable set $\phi(x,b)$ has (dimension, measure) $(d,r)$ if $\mathbf{F}\models\theta_{(d,r)}(b)$, where $\theta_{(d,r)}(y)$ is given by Fact \ref{psf}; note that, provided $\mathbf{F}$ is sufficiently large, the choice of $(d,r)$ will be unique. Also, if $\mathbf{F}$ is pseudofinite, then it is proved in Propositions 4.9 and 5.3 of \cite{chatzidakis_vdd_macintyre} that a definable set has dimension $d$ in the above sense if and only if the algebraic dimension of its Zariski closure is $d$, and that field-theoretic algebraic closure in $\mathbf{F}$ coincides with model-theoretic algebraic closure.

One has the following easy consequence of Fact \ref{psf} and the triangle inequality; the case $q=0$ occurs if $A$ has smaller dimension than $B$. The proof is easy but we include it to demonstrate moving back and forth between statements in terms of complexity of definable sets and statements in terms of individual first-order formulas.

\begin{corollary}\label{chatzidakis_same_dimension}
    For any $M$, there is a positive constant $C>0$ and a finite set of rationals $Q\subseteq[0,1]$, with the following properties. Suppose $\mathbf{F}$ is a finite field, and that $A\subseteq B$ are definable sets in $\mathbf{F}$ of complexity at most $M$. Then there is some $q\in Q$ such that $||A|-q|B||\leqslant C|\mathbf{F}|^{-1/2}|B|$, where $q$ depends only on the (dimension, measure) of $A$ and $B$ and $q=0$ if $A$ has strictly smaller dimension than $B$.
\end{corollary}
\begin{proof}
    Fix $M$. It suffices to find a $C$ and $Q$ that work for sufficiently large $\mathbf{F}$, as then there are only finitely many possible exceptions and we can account for them by increasing $C$.

    Let $\Delta$ be the set of all partitioned formulas\footnote[1]{Ie, formulas $\phi(x,y)$ whose free variables have been split into a distinguished choice of `object tuple' $x$ and `parameter tuple' $y$.} in the language of rings of length at most $M$. For each $\phi(x,y)\in\Delta$, let $C_\phi\geqslant 0$, $D_\phi\subseteq\mathbb{N}\times\mathbb{Q}_{>0}$, and $(\theta_{(d,r)}^\phi(y):(d,r)\in D_\phi\}$ witness Fact \ref{psf} for the formula $\phi(x,y)$. For each $\phi\in\Delta$, let $R_\phi\subset\mathbb{Q}_{>0}$ be the set of second coordinates appearing in $D_\phi$ and let $R=\bigcup_{\phi\in\Delta}R_\phi$. Let $C_0=\max_{\phi\in\Delta}C_\phi$ and let $Q$ be the set containing $0$ and also all rationals of form $r_1/r_2$, where $r_1\leqslant r_2$ lie in $R$. We claim that $Q$ and $C:=2C_0^2(\max(R)+1)$ witness the Corollary.

    To see this, fix a finite field $\mathbf{F}$ and sets $A\subseteq B$ in $\mathbf{F}$ of complexity at most $M$; thus there are formulas $\phi(x,y),\psi(x,z)\in\Delta$ and tuples $b\in\mathbf{F}^{|y|},c\in\mathbf{F}^{|z|}$, with $A=\phi(\mathbf{F}^{|x|},b)$ and $B=\psi(\mathbf{F}^{|x|},c)$. If $A$ is empty, then we may pick $q=0\in Q$ and have the desired inequality, so we may assume $A\neq\varnothing$ and thus that $B\neq\varnothing$ as well.
    
    Thus pick $(d,r)\in D_\phi$ and $(e,s)\in D_\psi$ such that $\mathbf{F}\models\theta^\phi_{(d,r)}(b)$ and $\mathbf{F}\models\theta^\psi_{(e,s)}(c)$. Now, since $A,B\neq\varnothing$, by Fact \ref{psf}(2) we have $$\left||A|-r|\mathbf{F}|^d\right|\leqslant C_\phi|\mathbf{F}|^{d-1/2}\text{ and }\left||B|-s|\mathbf{F}|^e\right|\leqslant C_\psi|\mathbf{F}|^{e-1/2}.$$ On the other hand, by Fact 2.9(3), also $|\mathbf{F}|^e/C_\psi\leqslant|B|$. Since $A\subseteq B$, taking $\mathbf{F}$ to be sufficiently large we may assume that $d\leqslant e$ and that, if $d=e$, then that $r\leqslant s$. If $d<e$, then also $d<e-1/2$, so that
    \begin{align*}
        |A|&\leqslant r|\mathbf{F}|^d+C_\phi|\mathbf{F}|^{-1/2}|\mathbf{F}|^{d}\leqslant r|\mathbf{F}|^{-1/2}|\mathbf{F}|^e+C_\phi|\mathbf{F}|^{-1/2}|\mathbf{F}|^e\\
        &= (r+C_\phi)|\mathbf{F}|^{-1/2}|\mathbf{F}|^e\leqslant (r+C_\phi)C_\psi|\mathbf{F}|^{-1/2}|B|\leqslant C|\mathbf{F}|^{-1/2}|B|,
    \end{align*} and we may take $q=0$.

    Otherwise suppose $d=e$. By Fact 2.9(2), $|A|$ lies within $C_\phi|\mathbf{F}|^{d-1/2}$ of $r|\mathbf{F}|^d$ and $r|\mathbf{F}|^d$ lies within $(r/s)C_\psi|\mathbf{F}|^{d-1/2}$ of $(r/s)|B|$. So by the triangle inequality, and the fact that $|\mathbf{F}|^{d-1/2}\leqslant C_\psi|\mathbf{F}|^{-1/2}|B|$, we have $$\left||A|-(r/s)|B|\right|\leqslant (C_\phi+(r/s)C_\psi)C_\psi|\mathbf{F}|^{-1/2}|B|\leqslant C|\mathbf{F}|^{-1/2}|B|;$$ since $r/s\in Q$ we are done.
\end{proof}

Now, let $\mathfrak{C}$ be a saturated pseudofinite field. Let $U$ be a non-empty $\mathfrak{C}$-definable set, and let $(d,s)$ be the (dimension, measure) of $U$. Given a $\mathfrak{C}$-definable set $X\subseteq U$, define $\nu_U(X)$ to be $0$ if $\dim(X)<\dim(U)$, and to be $r/s$ if the (dimension, measure) of $X$ is $(\dim(U),r)$. Then $\nu_U$ is a Keisler measure on $U$, meaning that it is a finitely additive probability measure on the Boolean algebra of $\mathfrak{C}$-definable subsets of $U$, and it is invariant under all $\mathfrak{C}$-definable bijections fixing $U$ setwise. Now the following is Lemma 1.1 in \cite{pillay_starchenko}:

\begin{fact}\label{indep_reals_const}
Fix a small model $M\prec\mathfrak{C}$ over which $U$ is definable. Suppose $\phi(x,y)$ and $\psi(x,z)$ are two $L(M)$-formulas such that $\phi(x,y)\to x\in U$ and $\psi(x,z)\to x\in U$. Then, for any $q\in S_y(M)$ and $r\in S_z(M)$, the value $\nu_U(\phi(x,b)\wedge\psi(x,c))$ is constant as $(b,c)$ ranges over $\{(b,c):b\models q,c\models r,b\ind_M c\}$.
\end{fact}


\section{Finding a good subgroup}\label{infinitary_section}
In this section we will prove the `group analogue' of Theorem 27 in \citep{tao}. We will work with the conventions of Section \ref{psf_section}. Thus let $\mathfrak{C}$ be a saturated pseudofinite field, let $M\prec\mathfrak{C}$ be a small model, and let $G$ be an $M$-definable group. Throughout we will write $\nu=\nu_G$, as defined in Section \ref{psf_section}, so that $\nu$ is a bi-translation-invariant and inversion-invariant Keisler measure on $G$.

\begin{lemma}\label{pqr strong}
Suppose that $\mathfrak{d},\mathfrak{e}$ are cosets of $G^{00}_M$, and that $q,q'\in S_G(M)$ and $r,r'\in S_G(M)$ are f-generic types concentrated on $\mathfrak{d}$ and $\mathfrak{e}$ respectively. Then there is $a\in G(\mathfrak{C})$ and $b\models q,c\models r$ such that $b\ind_M c$ and $ab\ind_M ac$ and $ab\models q'$ and $ac\models r'$.
\end{lemma}
\begin{proof}
    Let $p\in S_G(M)$ be any f-generic type concentrated on $G^{00}_M$. Let $b\models q$ and $c\models r$ be $M$-independent realizations of $q,r$. By Fact \ref{pq=r}, there are $u,v\models p$ such that $ub\models q'$ and $vc\models r'$ and such that each triple $(u,b,ub)$ and $(v,c,vc)$ is pairwise $M$-independent. Then $\tp(u/M,b)$ and $\tp(v/M,c)$ are each non-forking extensions of $p$, and $b\ind_M c$, so by the independence theorem we may find $a$ such that $a\equiv_{(M,b)}u$ and $a\equiv_{(M,c)}v$ and $a\ind_M(b,c)$. By the first two conditions, we have $ab\models q'$ and $ac
    \models r'$, so we just need to show $ab\ind_M ac$.

    Since $b\ind_M c$ and $a\ind_M(b,c)$, we have $b\ind_M(a,c)$. By Fact \ref{f-gen_facts}(4,5) it follows that $ab\ind_M(a,c)$, and hence in particular that $ab\ind_M ac$.
\end{proof}

\begin{lemma}\label{const_on_generics}
Let $D\subseteq G$ be an $M$-definable set, and let $\mathfrak{d},\mathfrak{e}$ be cosets of $G^{00}_M$. Then the value $\nu(bD\cap cD)$ is constant as $(b,c)$ ranges over the set of all $(b,c)\in\mathfrak{d}\times\mathfrak{e}$ such that $b\ind_M c$ and such that $\tp(b/M)$ and $\tp(c/M)$ are f-generic.
\end{lemma}
\begin{proof}
Fix f-generic types $q',r'\in S_G(M)$ concentrated on $G^{00}_M,\mathfrak{c},\mathfrak{d}$ respectively. By Fact \ref{indep_reals_const}, the value $s:=\nu(uD\cap vD)$ is constant as $(u,v)$ ranges over $$\{(u,v):u\models q',v\models r',u\ind_M v\}.$$ We claim that, for any $(b,c)\in\mathfrak{d}\times\mathfrak{e}$ such that $b\ind_M c$ and such that $\tp(b/M)$ and $\tp(c/M)$ are f-generic, we have $\nu(bD\cap cD)=s$. To see this, fix any f-generic types $q,r\in S_G(M)$ concentrated on $\mathfrak{d},\mathfrak{e}$ respectively; by Fact \ref{indep_reals_const}, and since $q,r$ are arbitrary, it suffices to find $b\models q$ and $c\models r$ with $b\ind_M c$ and $\nu(bD\cap cD)=s$.

By Lemma \ref{pqr strong}, there are $a\in G(\mathfrak{C}),b\models q,c\models r$ such that $b\ind_Mc$ and $ab\ind_M ac$ and $ab\models q'$ and $ac\models r'$. But now $s=\nu(abD\cap acD)$ since $ab,ac$ are $M$-independent realizations of $q',r'$, and $\nu(abD\cap acD)=\nu(a(bD\cap cD))=\nu(bD\cap cD)$ by translation invariance of $\nu$.
\end{proof}

Now the following is proved as in Corollary 1.2 in \cite{pillay_starchenko}.

\begin{lemma}\label{HF_fixed_cosets}
    Let $D\subseteq G$ be an $M$-definable set, and fix $v,w\in G(\mathfrak{C})$. Then there is an $M$-definable set $F\subseteq G\times G$ with $\dim(F)<2\dim(G)$, and an $M$-definable finite-index normal subgroup $H\leqslant G$, such that $\nu(bD\cap cD)$ is constant as $(b,c)$ ranges over $(vH\times wH)\setminus F$.
\end{lemma}

\begin{proof}
    Let $\mathfrak{d},\mathfrak{e}$ denote the cosets $vG^{00}_M,wG^{00}_M$, respectively. Let $(H_i)_{i\in I}$ be the family of $M$-definable finite-index normal subgroups of $G$, so that by Fact \ref{g0=g00} $G^{00}_M=\bigcap_{i\in I}H_i$. For each $i\in I$, let $v_i,w_i\in G(M)$ be such that $v_iH_i=vH_i$ and $w_iH_i=wH_i$. Now consider the partial type $\pi(x,y)$ declaring $\{x\in v_iH_i\wedge y\in w_iH_i:i\in I\}$ and $$\{\neg\theta(x,y):\theta(x,y)\in L(M),\dim(\theta(x,y))<2\dim(G)\}.$$ If $(b,c)\models\pi(x,y)$, then $b\in\mathfrak{d}$ and $c\in\mathfrak{e}$ and $b\ind_M c$, and by Lemma \ref{f-gen_iff_max_dim} $\tp(b/M)$ and $\tp(c/M)$ are f-generic; in other words $b$ and $c$ then satisfy the hypotheses of Lemma \ref{const_on_generics} for the cosets $\mathfrak{d},\mathfrak{e}$. Let $s$ be the value of $\nu(bD\cap cD)$ for some $(b,c)\models\pi(x,y)$. By Fact \ref{psf}, there is a formula $\delta(x,y)$ expressing that $\nu(xD\cap yD)=s$, and by Lemma \ref{const_on_generics} we have $\pi(x,y)\vdash\delta(x,y)$, so the claim follows from compactness.
\end{proof}

\begin{lemma}\label{HF_left}
    Let $D\subseteq G$ be an $M$-definable set. Then there is an $M$-definable finite-index normal subgroup $H\leqslant G$ and an $M$-definable set $F\subseteq G\times G$ with $\dim(F)<2\dim(G)$ and such that, for any cosets $V,W$ of $H$, the value $\nu(bD\cap cD)$ is constant as $(b,c)$ ranges over $(V\times W)\setminus F$.
\end{lemma}
\begin{proof}
    Let $(g_i)_{i\in I}$ be elements of $G(\mathfrak{C})$ representing all the cosets of $G^{00}_M$. For each $i,j\in I$, let $F_{ij}$ and $H_{ij}$ be given by Lemma \ref{HF_fixed_cosets} with $v=g_i$ and $w=g_j$.
    
    Now, fix $j\in I$. Then $G$ is covered by $(g_iH_{ij})_{i\in I}$, so by compactness there are $i_1,\dots,i_n$ with $G=g_{i_1}H_{i_1j}\cup\dots\cup g_{i_n}H_{i_nj}$. Let $H_j=H_{i_1j}\cap\dots\cap H_{i_nj}$ and $F_j=F_{i_1j}\cup\dots\cup F_{i_nj}$. Then $H_j$ is still finite-index in $G$ and $F_j$ still has dimension $<2\dim(G)$. Moreover, for any $g\in G(\mathfrak{C})$, the value $\nu(bD\cap cD)$ is constant as $(b,c)$ ranges over $(gH_j\times g_jH_j)\setminus F$; indeed, there is some $1\leqslant k\leqslant n$ with $g\in g_{i_k}H_{i_kj}$, so that $gH_j\subseteq g_{i_k}H_{i_kj}$, and since $F_{i_kj}\subseteq F_j$ thus $(gH_j\times g_jH_j)\setminus F_j\subseteq (g_{i_k}H_{i_kj}\times g_jH_{i_kj})\setminus F_{i_kj}$.

    Now, $G$ is again covered by $(g_jH_j)_{j\in I}$. So again by compactness there are $j_1,\dots,j_m\in I$ with $G=g_{j_1}H_{j_1}\cup\dots\cup g_{j_m}H_{j_m}$. By an argument as in the previous paragraph, letting $H=H_{j_1}\cap\dots\cap H_{j_m}$ and $F=F_{j_1}\cup\dots\cup F_{j_m}$ gives the desired result.
\end{proof}

\begin{corollary}\label{HF_bi}
    Let $D\subseteq G$ be an $M$-definable set. Then there is an $M$-definable finite-index normal subgroup $H\leqslant G$ and an $M$-definable set $F\subseteq G\times G$ with $\dim(F)<2\dim(G)$ and such that, for any cosets $V,W$ of $H$, the values $\nu(b^{-1}D\cap c^{-1}D)$ and $\nu(Db\cap Dc)$ are constant as $(b,c)$ ranges over $(V\times W)\setminus F$.
\end{corollary}

\begin{proof}
    Let $H_1,F_1$ and $H_2,F_2$ be given by Lemma \ref{HF_left} for $D$ and $D^{-1}=\{d^{-1}:d\in D\}$, respectively. Let $H=H_1\cap H_2$ and $F=F_1^{-1}\cup F_2^{-1}$, where $F_i^{-1}=\{(u^{-1},v^{-1}):(u,v)\in F_i\}$; note that $\dim(F_i^{-1})=\dim(F_i)$ and so $\dim(F)<2\dim(G)$.


    Let $V,W$ be cosets of $H$, and for each $i\in\{1,2\}$ let $V_i,W_i$ be the cosets of $H_i$ with $V\subseteq V_i$ and $W\subseteq W_i$.
    
    Now we have $(V\times W)\setminus F\subseteq ((V_1^{-1}\times W_1^{-1})\setminus F_1)^{-1}$. Thus, as $(b,c)$ ranges over $(V\times W)\setminus F$, $(b^{-1},c^{-1})$ ranges over $(V_1^{-1}\times W_1^{-1})\setminus F_1$, and hence (since $V_1^{-1}$ and $W_1^{-1}$ are cosets of $H_1$) the value of $\nu(b^{-1}D\cap c^{-1}D)$ remains constant.

    Symmetrically, as $(b,c)$ ranges over $(V\times W)\setminus F$, the value of $\nu(b^{-1}D^{-1}\wedge c^{-1}D^{-1})$ remains constant. But by inversion-invariance of $\nu$ this latter value is equal to $\nu(Db\wedge Dc)$, as needed.
\end{proof}

\section{Graph regularity}
We get the following as a consequence of Corollary \ref{HF_bi} and Fact \ref{psf}; the proof is routine.

\begin{corollary}\label{HF_finitary}
    Fix $M>0$. Then there is $C=C(M)>0$ such that the following holds for any finite field $\mathbf{F}$. If $G$ is a definable group in $\mathbf{F}$ of complexity $\leqslant M$, and $D\subseteq G$ is a definable subset of $G$ of complexity $\leqslant M$, then there is a definable normal subgroup $H\leqslant G$, of complexity at most $C$, with $[G:H]\leqslant C$ and with the following properties:
    \begin{enumerate}
        \item For any cosets $V,W$ of $H$, there is a rational number $r\leqslant 1$ such that, for all but $\leqslant C|\mathbf{F}|^{-1/2}|G|^2$ many pairs $(b,c)\in V\times W$, $\big||b^{-1}D\cap c^{-1}D|-r|G|\big|\leqslant C|\mathbf{F}|^{-1/2}|G|$.
        \item For any cosets $V,W$ of $H$, there is a rational number $r\leqslant 1$ such that, for all but $\leqslant C|\mathbf{F}|^{-1/2}|G|^2$ many pairs $(b,c)\in V\times W$, $\big||Db\cap Dc|-r|G|\big|\leqslant C|\mathbf{F}|^{-1/2}|G|$.
    \end{enumerate}
 \end{corollary}
 \begin{proof}
    Suppose the theorem does not hold. Thus in particular, for every $n\in\mathbb{N}$, taking $C=n$ does not work; picking data witnessing the failure of this, and applying the pigeonhole principle (since there are only finitely many formulas in the language of rings of length at most $M$), we may thus find formulas $\phi(x,y),\psi(x,y),\chi(x_1,x_2,x_3,y)$ of length at most $M$, prime powers $q_n,n\in\mathbb{N}$, and tuples $b_n\in\mathbf{F}_{q_n}^{|y|}$, such that $\chi(x_1,x_2,x_3,b_n)$ defines a group operation on $G_n:=\phi(\mathbf{F}_{q_n}^{|x|},b_n)$, $D_n:=\psi(\mathbf{F}_{q_n}^{|x|},b_n)$ is a subset of $G_n$, and there are no definable subgroups of $G_n$ of complexity at most $n$ that have the desired properties with respect to $D_n$ and $C=n$. Again by pigeonhole, we may assume that all $G_n$ (respectively all $D_n$) have the same (dimension, measure), by the finiteness part of Fact \ref{psf}.

    Let $\mathbf{F}$ be the ultraproduct of the $\mathbf{F}_{q_n}$ along any non-principal ultrafilter. Note that $\mathbf{F}$ must be infinite; otherwise the $q_n$ and hence the $|G_n|$ would all be bounded above, and then for some sufficiently large $n$ taking the subgroup $G_n$ itself would then give the desired properties for $C=n$. Let $\mathfrak{C}\succ\mathbf{F}$ be a saturated model and let $b\in\mathbf{F}$ be the equivalence class of $(b_n)_{n\in\omega}$ in the ultraproduct. Let $G$ be the $\mathbf{F}$-definable group given by $\phi(x,b)$ and $\chi(x_1,x_2,x_3,b)$ and let $D\subseteq G$ be defined by $\psi(x,b)$. Let $H$ and $F$ be the $\mathbf{F}$-definable sets given by Corollary \ref{HF_bi} for $G$ and $D$, say defined by formulas $\widehat{\phi}(x,d)$ and $\theta(x_1,x_2,d)$, respectively, where $d$ is the equivalence class of $(d_n)_{n\in\omega}$ in the ultraproduct. Let $N$ be the maximum of the lengths of the formulas $\widehat{\phi}$ and $\theta$, and let $C,Q$ be given by Corollary \ref{chatzidakis_same_dimension} for $N$.

    By Łoś's theorem and Fact \ref{psf}, there are ultrafilter-many $n$ such that:
    \begin{enumerate}
    \item $\widehat{\phi}(x,d_n)$ defines a normal subgroup $H_n$ of $G_n$ of index $[G:H]$.
    \item $\theta(x_1,x_2,d_n)$ defines a subset $F_n$ of $G_n\times G_n$ of dimension $<2\dim(G_n)$.
    \item For any cosets $V,W$ of $H_n$, either $b^{-1}D_n\cap c^{-1}D_n$ has strictly smaller dimension than $G_n$ for all $(b,c)\in (V\times W)\setminus F_n$, or it has the same dimension as $G_n$ and constant measure as $(b,c)$ ranges over $(V\times W)\setminus F$.
    \item For any cosets $V,W$ of $H_n$, either $D_nb\cap D_nc$ has strictly smaller dimension than $G_n$ for all $(b,c)\in (V\times W)\setminus F_n$, or it has the same dimension as $G_n$ and constant measure as $(b,c)$ ranges over $(V\times W)\setminus F$.
    \end{enumerate} By Corollary \ref{chatzidakis_same_dimension} and item 2 above, $|F_n|\leqslant C|\mathbf{F}_{q_n}|^{-1/2}|G_n|^2$ for all such $n$. Moreover, by Corollary \ref{chatzidakis_same_dimension} and items 3 and 4 above, we have
    \begin{enumerate}
        \item For any cosets $V,W$ of $H_n$, there is a rational number $r\leqslant 1$, possibly $0$, such that, for all $(b,c)\in (V\times W)\setminus F_n$, $\big||b^{-1}D_n\cap c^{-1}D_n|-r|G_n|\big|\leqslant C|\mathbf{F}_{q_n}|^{-1/2}|G_n|$.
        \item For any cosets $V,W$ of $H_n$, there is a rational number $r\leqslant 1$, possibly $0$, such that, for all $(b,c)\in (V\times W)\setminus F_n$, $\big||D_nb\cap D_nc|-r|G_n|\big|\leqslant C|\mathbf{F}_{q_n}|^{-1/2}|G_n|$
    \end{enumerate}for all such $n$. Picking such an $n$ that is larger than $\max\{C,[G:H],N\}$, we obtain the desired contradiction.
 \end{proof}
 
We can now deduce the group analogue of Tao's `algebraic regularity lemma' from Corollary \ref{HF_finitary} in the same fashion as Tao deduces it from Proposition 27 in \cite{tao}.

\begin{lemma}\label{algebraic_regularity_lemma}
    Fix any $M>0$. Then there is $C=C(M)>0$ such that the following holds for any finite field $\mathbf{F}$. If $G$ is a definable group in $\mathbf{F}$ of complexity $\leqslant M$, and $D\subseteq G$ is a definable subset of $G$ of complexity $\leqslant M$, then there is a normal subgroup $H\leqslant G$ with $[G:H]\leqslant C$ and such that the following holds. For any cosets $V,W$ of $H$, the graph $(V,W,xy^{-1}\in D)$ is weakly $C|\mathbf{F}|^{-1/4}$-regular; ie, for any $A\subseteq V,B\subseteq W$, the value $|\{(a,b)\in A\times B:ab^{-1}\in D\}|$ differs from $$|\{(v,w)\in V\times W:vw^{-1}\in D\}|\frac{|A||B|}{|V||W|}$$ by at most $C|\mathbf{F}|^{-1/4}|H|^2$.
\end{lemma}
\begin{proof}
    Fix $M>0$ and let $C_1$ be given by Corollary \ref{HF_finitary}. Suppose that $\mathbf{F}$ is a finite field, that $G$ is a definable group in $\mathbf{F}$ of complexity $\leqslant M$, and that $D\subseteq G$ is a definable set of complexity $\leqslant M$. Let $H$ be given by Corollary \ref{HF_finitary}.

\begin{claim}\label{mean_0_individual}
    There is a constant $C_2$ depending only on $M$ such that, for any $U$ a coset of $H$, and any $f:U\to\mathbb{R}$ of mean $0$ and with $||f||_\infty\leqslant 1$,
    \begin{align*}
        &\sum_{g\in G}\left(\sum_{u\in U}1_{[ug\in D]}f(u)\right){}^2\text{ and }\sum_{g\in G}\left(\sum_{u\in U}1_{[gu^{-1}\in D]}f(u)\right){}^2
    \end{align*}are each bounded above by $C_2|\mathbf{F}|^{-1/2}|H|^3$.
\end{claim}
\begin{proof}
    We show the first bound, as the second bound is proved symmetrically using condition (2) of Corollary \ref{HF_finitary}; let $\Delta$ be the quantity we wish to bound. Expanding out the square in the sum gives $$\Delta=\sum_{g\in G}\sum_{u,v\in U}1_{[ug\in D]}1_{[vg\in D]}f(u)f(v),$$ which is just $\sum_{(u,v)\in U\times U}|u^{-1}D\cap v^{-1}D|f(u)f(v)$. By condition (1) of Corollary \ref{HF_finitary}, there is some rational $r\leqslant 1$ such that the set $$F:=\{(u,v)\in U\times U:\left||u^{-1}D\cap v^{-1}D|-r|G|\right|>C_1|\mathbf{F}|^{-1/2}|G|\}$$ has size $\leqslant C_1|\mathbf{F}|^{-1/2}|G|^2$. Now, since $||f||_\infty\leqslant 1$, $\big|r|G|f(u)f(v)-|u^{-1}D\cap v^{-1}D|f(u)f(v)\big|$ is bounded above by $C_1|\mathbf{F}|^{-1/2}|G|$ for all $(u,v)\in (U\times U)\setminus F$, whence $$\bigg|\sum_{(u,v)\in (U\times U)\setminus F}r|G|f(u)f(v)-\sum_{(u,v)\in (U\times U)\setminus F}|u^{-1}D\cap v^{-1}D|f(u)f(v)\bigg|$$ is bounded above by $|(U\times U)\setminus F|C_1|\mathbf{F}|^{-1/2}|G|\leqslant C_1|\mathbf{F}|^{-1/2}|G|^3\leqslant C_1^4|\mathbf{F}|^{-1/2}|H|^3$, where in the last inequality we used that $|G|\leqslant C_1|H|$.
    
    Also, since $D\subseteq G$ and $r\leqslant 1$, we have $\left||u^{-1}D\cap v^{-1}D|-r|G|\right|\leqslant |G|$ for all $(u,v)\in G\times G$, and in particular for all $(u,v)\in F$, so that $$\bigg|\sum_{(u,v)\in F}r|G|f(u)f(v)-\sum_{(u,v)\in  F}|u^{-1}D\cap v^{-1}D|f(u)f(v)\bigg|$$ is bounded above by $|F||G|\leqslant C_1|\mathbf{F}|^{-1/2}|G|^2|G|\leqslant C_1^4|\mathbf{F}|^{-1/2}|H|^3$, where in the last inequality we again used that $|G|\leqslant C_1|H|$.
    
    Taking $C_2=2C_1^4$ and applying the triangle inequality to the two estimates above gives that $\Delta$ differs from $\sum_{(u,v)\in U\times U}r|G|f(u)f(v)$ by at most $C_2|\mathbf{F}|^{-1/2}|H|^3$. But $f$ has mean $0$ so the latter sum is $0$ and the claim follows.
\end{proof}

\begin{claim}\label{mean_0_both}
    There is a constant $C_3$ depending only on $M$ such that, for any cosets $V,W$ of $H$, and any $f:V\to\mathbb{R}$ and $g:W\to\mathbb{R}$ with $||f||_\infty,||g||_\infty\leqslant 1$, if at least one of $f,g$ has mean $0$ then $|\sum_{(v,w)\in V\times W}1_{[vw^{-1}\in D]}f(v)g(w)|\leqslant C_3|\mathbf{F}|^{-1/4}|H|^2$.
\end{claim}

\begin{proof}
    Suppose for instance that $f$ has mean $0$; the other case is symmetric using the other half of Claim \ref{mean_0_individual}. We have
    \begin{align*}
    \left(\sum_{(v,w)\in V\times W}f(v)g(w)1_{[vw^{-1}\in D]}\right){}^2&=\left(\sum_{w\in W}\left(\sum_{v\in V}f(v)g(w)1_{[vw^{-1}\in D]}\right)\right){}^2\\
    &\leqslant |W|\sum_{w\in W}\left(\sum_{v\in V}f(v)g(w)1_{[vw^{-1}\in D]}\right){}^2 \\
    &=|W|\sum_{w\in W}g(w)^2\left(\sum_{v\in V}f(v)1_{[vw^{-1}\in D]}\right){}^2 \\
    &\leqslant |W|\sum_{w\in W}\left(\sum_{v\in V}f(v)1_{[vw^{-1}\in D]}\right){}^2 \\
    &\leqslant |W|\sum_{w\in G}\left(\sum_{v\in V}f(v)1_{[vw^{-1}\in D]}\right){}^2
\end{align*} where the first inequality follows from Cauchy-Schwarz, and the second follows from the fact that $||g||_\infty\leqslant 1$. But now $f$ has mean $0$, so that by the first bound in Claim \ref{mean_0_individual} the left hand side is bounded above by $|W|C_2|\mathbf{F}|^{-1/2}|H|^3= C_2|\mathbf{F}|^{-1/2}|H|^4$. So taking $C_3=\sqrt{C_2}$ works.
\end{proof}

Let $C_3$ be given by Claim 2. Suppose that $V,W$ are cosets of $H$ and that $A\subseteq V$ and $B\subseteq W$. Define $f:V\to \mathbb{R}$ and $g:W\to\mathbb{R}$ by $f(v)=1_{A}(v)-|A|/|H|$ and $g(w)=1_{B}(w)-|B|/|H|$. By Claim 2, $\sum_{(v,w)\in V\times W}\frac{|A|}{|H|}g(w)1_{[vw^{-1}\in D]}$ and $\sum_{(v,w)\in V\times W}f(v)\frac{|B|}{|H|}1_{[vw^{-1}\in D]}$ and $\sum_{(v,w)\in V\times W}f(v)g(w)1_{[vw^{-1}\in D]}$ are all bounded in magnitude by $C_3|\mathbf{F}^{-1/4}||H|^2$. Also clearly $\sum_{(v,w)\in V\times W}\frac{|A|}{|H|}\frac{|B|}{|H|}1_{[vw^{-1}\in D]}=|\{(v,w)\in V\times W:vw^{-1}\in D\}|\frac{|A||B|}{|V||W|}$. But the sum of those four terms is 
\begin{align*}\sum_{(v,w)\in V\times W}\left(f(v)+\frac{|A|}{|H|}\right)\left(g(w)+\frac{|B|}{|H|}\right)1_{[vw^{-1}\in D]}&=\sum_{(v,w)\in V\times W}1_A(v)1_B(w)1_{[vw^{-1}\in D]}\\&=|\{(a,b)\in A\times B:ab\in D\}|.
\end{align*}Thus, applying the triangle inequality to the sum of the three terms bounded above, $|\{(a,b)\in A\times B:ab\in D\}|$ differs from $|\{(v,w)\in V\times W:vw^{-1}\in D\}|\frac{|A||B|}{|V||W|}$ by at most $3C_3|\mathbf{F}|^{-1/4}|H|^2$, and so taking $C=3C_3$ gives the result.
\end{proof}

\section{Sharper quasirandomness bounds}
By Fact \ref{qr_thm}, Lemma \ref{algebraic_regularity_lemma} immediately implies the following:

\begin{lemma}\label{qr_weak_bound}
    For any $M$, there is a positive constant $C>0$ such that the following holds. Suppose $\mathbf{F}$ is a finite field, and that $G$ is a definable group in $\mathbf{F}$ and $D\subseteq G$ is a definable subset, both of complexity at most $M$. Then there is a definable normal subgroup $H\leqslant G$, of index and complexity at most $C$, such that, for any cosets $V,W$ of $H$, the bipartite graph $(V,W,xy^{-1}\in D)$ is $C|\mathbf{F}|^{-1/4}$-quasirandom.
\end{lemma}

However, with some computation, the nature of the error bounds in Fact \ref{psf} allows us to get the following improvement; all hypotheses are the same, and the difference is the better bound on the degree of quasirandomness in the conclusion. We remark that nothing in the proof uses that we are working in the setting of definable groups, and the same argument also works in the setting of \cite{tao} to give $C|\mathbf{F}|^{-1/2}$-quasirandomness of the corresponding graphs there.

\begin{theorem}\label{main_theorem}
    For any $M$, there is a positive constant $K>0$ such that the following holds. Suppose $\mathbf{F}$ is a finite field, and that $G$ is a definable group in $\mathbf{F}$ and $D\subseteq G$ is a definable subset, both of complexity at most $M$. Then there is a definable normal subgroup $H\leqslant G$, of index and complexity at most $K$, such that, for any cosets $V,W$ of $H$, the bipartite graph $(V,W,xy^{-1}\in D)$ is $K|\mathbf{F}|^{-1/2}$-quasirandom.
\end{theorem}
\begin{proof}
    It suffices to find a $K$ that works for $\mathbf{F}$ sufficiently large, since then there are only finitely many possible exceptions and we can account for them by increasing $K$.

    First let $C$ be given by Lemma \ref{qr_weak_bound} for $M$. Now fix $\mathbf{F}$, $G$, and $D\subseteq G$ as in the theorem hypotheses, and let $H$ be given by the lemma. So $H$ has complexity at most $C$. Now fix $V,W$ cosets of $H$; then $V,W$ have complexity at most $O_C(1)$. Since $C=O_M(1)$ depends only on $M$, it follows that the following sets are definable of complexity $O_M(1)$; fix some $M'$ depending only on $M$ such that all of the following have complexity at most $M'$:
    \begin{align*}
        &(V\times W)^4=(V\times W)\times (V\times W)\times (V\times W)\times (V\times W), \\
        &X:=\{((v,w)\in V\times W:vw^{-1}\in D)\}\text{ and }X^4=X\times X\times X\times X,\\
        &Y:=\{(v_1,v_2,w_1,w_2):w_1,w_2\in W,v_1,v_2\in V\cap Dw_1\cap Dw_2\}.
    \end{align*} Now, we have $|V|=|W|=|H|$. So $|X|=\delta|H|^2$, where $\delta$ is the edge density of the graph $(V,W,xy^{-1}\in D)$, and hence $|X^4|=\delta^4|H|^8$.
    
    Also, $|Y|=\sum_{w_1,w_2\in W}|N_{w_1}\cap N_{w_2}|^2$, where $N_{w_i}\subseteq V$ is the set of $v\in V$ connected to $w_i$ in $(V,W,xy^{-1}\in D)$. By Lemma \ref{qr_weak_bound}, and the definition of quasirandomness, we hence have that $$\left||Y|-\delta^4|H|^4\right|\leqslant C|\mathbf{F}|^{-1/4}|H|^4$$ and hence $\left||Y|/|H|^4-\delta^4\right|\leqslant C|\mathbf{F}|^{-1/4}$; call this inequality (1).

    Now, let $C'$ and $Q\subseteq[0,1]$ be given for $M'$ by Fact \ref{chatzidakis_same_dimension}. By Fact \ref{chatzidakis_same_dimension} applied to $X^4\subseteq (V\times W)^4$ and $Y\subseteq V^2\times W^2$, and again using $|V|=|W|=|H|$, there are thus $r,s\in Q$ such that \begin{align*}
        \left||X|^4-r|H|^8\right|&\leqslant C'|\mathbf{F}|^{-1/2}|H|^8\text{ and}\ \\
        \left||Y|-s|H|^4\right|&\leqslant C'|\mathbf{F}|^{-1/2}|H|^4,
    \end{align*} whence $\left||X|^4/|H|^8-r\right|\leqslant C'|\mathbf{F}|^{-1/2}$, ie $|\delta^4-r|\leqslant C'|\mathbf{F}|^{-1/2}$, and $\left||Y|/|H|^4-s\right|\leqslant C'|\mathbf{F}|^{-1/2}$; call these latter two inequalities (2) and (3) respectively.

    Using inequalities (1), (2), and (3), and the triangle inequality, we have the bound $|r-s|\leqslant C|\mathbf{F}|^{-1/4}+2C'|\mathbf{F}|^{-1/2}$. On the other hand, $r,s$ are coming from a finite set $Q$ of rational numbers that depends only on $M'$ and hence only on $M$, and $C',C$ depend only on $M',M$ and hence only on $M$; thus, when $\mathbf{F}$ is sufficiently large, this forces $r=s$. So assume $\mathbf{F}$ is sufficiently large and hence that $r=s$.
    
    Our inequalities (2) and (3) now become $|\delta^4-r|\leqslant C'|\mathbf{F}|^{-1/2}$ and $\left||Y|/|H|^4-r\right|\leqslant C'|\mathbf{F}|^{-1/2}$; by the triangle inequality, thus $\left||Y|/|H|^4-\delta^4\right|\leqslant 2C'|\mathbf{F}|^{-1/2}$. But this means precisely that $(V,W,xy^{-1}\in D)$ is $2C'|\mathbf{F}|^{-1/2}$-quasirandom. Thus taking $K=2C'$ works for sufficiently large $\mathbf{F}$, and since $C'$ depended only on $M'$ and hence only on $M$ we are done.
\end{proof}

Now by Corollary \ref{qr_subsets_of_groups_2}, and noting that the graph $(H,H,xy^{-1}\in (gD\cap H))$ is isomorphic to the graph $(g^{-1}H,H,xy^{-1}\in D)$, we obtain the following.

\begin{corollary}\label{main_theorem_fourier}
    For any $M$, there is a positive constant $K>0$ such that the following holds. Suppose $\mathbf{F}$ is a finite field, and that $G$ is a definable group in $\mathbf{F}$ and $D\subseteq G$ is a definable subset, both of complexity at most $M$. Then there is a definable normal subgroup $H\leqslant G$, of index and complexity at most $K$, such that, for any $g\in G$, the intersection $Dg\cap H$ is a $K|\mathbf{F}|^{-1/8}$-quasirandom subset of $H$.
 \end{corollary}

 \section{Algebraic groups}\label{special_cases_sec}
In this section we discuss some special cases of Theorem \ref{main_theorem}, and remark on their connections with the results from \cite{green} and \cite{gowers_groups} mentioned in Section \ref{qr_prelim_section}. We will focus on algebraic groups, which are in any case the main definable groups of interest in finite fields; in fact, by Theorem C of \cite{hrushovski_pillay}, any definable group in a pseudofinite field $\mathbf{F}$ admits a virtual isogeny to the $\mathbf{F}$-points of an algebraic group defined over $\mathbf{F}$.

Throughout the above paper, we were using $G$ to refer to an abstract group, but now we will change notation and use $G$ to refer to an algebraic group, and use $G(\mathbf{F})$ to refer to the $\mathbf{F}$-points of $G$ if $\mathbf{F}$ is a field over which $G$ is defined. For convenience we will work with the `classical' (ie non-scheme-theoretic) perspective, so that we identify an algebraic variety $V$ over a field $\mathbf{F}$ with its $\mathbf{K}$-points $V(\mathbf{K})$, where $\mathbf{K}$ is an algebraically closed field containing $\mathbf{F}$. If $G$ is a linear algebraic group over a field $\mathbf{F}$, by the complexity of $G$ we mean the number of symbols in the system of polynomial equations defining $G$.

There is a general setting in which one does not need to pass to a subgroup in Theorem \ref{main_theorem} to obtain quasirandomness, which is that of simply connected algebraic groups. Let us briefly recall a few definitions. If $\mathbf{F}$ is a field, and $G_1,G_2$ are connected algebraic groups over $\mathbf{F}$, then a (central) isogeny over $\mathbf{F}$ from $G_1$ to $G_2$ is a morphism of algebraic groups over $\mathbf{F}$ from $G_1$ to $G_2$ that is surjective and has finite (central) kernel. If $G$ is a connected algebraic group over $\mathbf{F}$, then it is simply connected (as an algebraic group) if, for every connected algebraic group $G'$ over $\mathbf{F}$, any central isogeny $G'\to G$ over $\mathbf{F}$ is an isomorphism. For example, $\mathrm{SL}_2$ is simply connected as an algebraic group over any field.

From \cite{hrushovski_pillay_2} we have the following:

\begin{fact}\label{simply_connected_gives_def_connected}
    Suppose $G$ is a connected, simply connected algebraic group over a pseudofinite field $\mathbf{F}$. Then $G(\mathbf{F})$ is a definably connected group in $\mathbf{F}$, ie it has no proper finite-index subgroup definable in $\mathbf{F}$.
\end{fact}

The following is a standard consequence:

\begin{lemma}\label{simply connected gives definably connected, finite case}
    Fix a natural number $C$. Then there is some $N$ such that, for any prime power $q\geqslant N$, and any connected, simply connected linear algebraic group $G$ over $\mathbf{F}_q$ of complexity at most $C$, if $H\leqslant G(\mathbf{F}_q)$ is of index at most $C$ and definable in $\mathbf{F}_q$ of complexity at most $C$, then $H=G(\mathbf{F}_q)$.
\end{lemma}
\begin{proof}
    Otherwise we may find an increasing sequence $(q_n)_{n\in\omega}$ prime powers, and for each $n$ a connected, simply connected algebraic group $G_n$ over $\mathbf{F}_{q_n}$ of complexity at most $C$, with a proper subgroup $H_n< G_n(\mathbf{F}_{q_n})$ of index at most $C$ and definable in $\mathbf{F}_{q_n}$ of complexity at most $C$. Applying the pigeonhole principle, we may assume that all of the $G_n$ are defined by systems of polynomial equations of the same shape (with different coefficients), and that all of the $H_n$ are defined by instances of the same formula.

    Now, let $\mathbf{F}$ be a non-principal ultraproduct of the $\mathbf{F}_{q_n}$, let $H$ be the corresponding ultraproduct of the $H_n$, and let $\mathbf{K}$ be the corresponding ultraproduct of the algebraically closed fields $\mathbf{F}_{q_n}^{\mathrm{alg}}$. Consider $\mathbf{F}$ as a substructure of $\mathbf{K}$ in the natural way. By Łoś's theorem applied in $\mathbf{K}$, if we let $G$ be defined by the same polynomials as the $G_n$ are and with coefficients coming from the elements of $\mathbf{F}$ given by the equivalence classes of the sequences of elements of $\mathbf{F}_{q_n}$ that give the coefficients of the $G_n$, then $G$ is an algebraic group over $\mathbf{F}$. Also, by Łoś's theorem applied in $\mathbf{F}$, $H$ is a definable proper subgroup of $G(\mathbf{F})$ of index at most $C$. So by Fact \ref{simply_connected_gives_def_connected} it suffices to show that $G$ is connected and simply connected to get a contradiction.
    
    This follows from general grounds. For example, suppose that $G$ is not simply connected; thus there is an algebraic group $G'$ and a central isogeny $f:G'\to G$, all defined over $\mathbf{F}$, with $f$ not an isomorphism. Thus $f$ defines a surjective map $G'(\mathbf{K})\to G(\mathbf{K})$ with finite non-trivial central kernel, say of size $k>1$. Now by Łoś's theorem in $\mathbf{K}$, there are ultrafilter-many $n$ such that, if we replace the coefficients in the polynomials over $\mathbf{F}$ that define $f$ and $G'$ by choices of corresponding elements of $\mathbf{F}_{q_n}$, we obtain an algebraic group $G'_n$ over $\mathbf{F}_{q_n}$ and a morphism $f_n:G'_n\to G_n$ over $\mathbf{F}_{q_n}$. Similarly, again by Łoś's theorem in $\mathbf{K}$, there are ultrafilter-many of those $n$ such that $f_n$ defines a surjection from $G'_n(\mathbf{F}_{q_n}^{\mathrm{alg}})$ to $G_n(\mathbf{F}_{q_n}^{\mathrm{alg}})$ with central kernel of size $k$, contradicting that the $G_n$ are simply connected. That $G$ is connected as an algebraic group follows similarly.
\end{proof}

From the above lemma and Theorem \ref{main_theorem}, we immediately obtain the following:

\begin{proposition}\label{simply connected theorem}
    For any $M>0$, there is a constant $C>0$ such that the following holds. Suppose $\mathbf{F}$ is a finite field, and that $G$ is a connected, simply connected linear algebraic group over $\mathbf{F}$ of complexity at most $M$. Then, for any subset $D\subseteq G(\mathbf{F})$ definable in $\mathbf{F}$ of complexity at most $M$, the bipartite graph $(G(\mathbf{F}),G(\mathbf{F}),xy^{-1}\in D)$ is $C|\mathbf{F}|^{-1/2}$-quasirandom.
\end{proposition}

Below we will discuss generalizations of this in two somewhat orthogonal directions.

\subsection{Sufficiently large characteristic}
Suppose $G$ is an algebraic group over $\mathbb{Z}$, so that $G$ may be considered as an algebraic group over over any field. There are many natural cases where $G$ is not simply connected over any field of positive characteristic, but is simply connected over any field of characteristic $0$. For example, it is easy to check that the additive group over a field of characteristic $0$ is simply connected, whereas in positive characteristic the Frobenius map is a non-invertible isogeny from the additive group to itself.

In such a case, one can get the analogue of Lemma \ref{simply connected gives definably connected, finite case} for finite fields of sufficiently large characteristic; one argues exactly as in Lemma \ref{simply connected gives definably connected, finite case}, just applying Fact \ref{simply_connected_gives_def_connected} in the pseudofinite field of characteristic $0$ obtained by taking an ultraproduct of counterexamples of increasing characteristic.

\begin{lemma}
    Let $G$ be a linear algebraic group defined over $\mathbb{Z}$ such that $G$ is connected and simply connected over a field of characteristic $0$, and let $C$ be a natural number. Then there is some $N$ such that, for any finite field $\mathbf{F}$ of characteristic at least $N$, $G(\mathbf{F})$ has no proper subgroup definable in $\mathbf{F}$ of index and complexity at most $C$.
\end{lemma}

So one gets the corresponding analogue of Proposition \ref{simply connected theorem}:

\begin{proposition}\label{simply connected theorem char 0}
    Let $G$ be a linear algebraic group defined over $\mathbb{Z}$ such that $G$ is connected and simply connected over a field of characteristic $0$, and let $M$ be a natural number. Then there is some $C$ such that the following holds: if $\mathbf{F}$ is a finite field of characteristic at least $C$, then, for any subset $D\subseteq G(\mathbf{F})$ definable in $\mathbf{F}$ of complexity at most $M$, the bipartite graph $(G(\mathbf{F}),G(\mathbf{F}),xy^{-1}\in D)$ is $C|\mathbf{F}|^{-1/2}$-quasirandom.
\end{proposition}

There are many natural examples; let us mention two in particular.

 \subsubsection{Additive group}\label{additive group section}
As mentioned to above, the additive group is simply connected over a field of characteristic $0$, so Proposition \ref{simply connected theorem char 0} applies to it. So in particular we get the following:

 \begin{proposition}\label{char_0}
     For any $M>0$, there is a constant $C>0$ such that the following holds. Suppose $\mathbf{F}$ is a finite field of characteristic $>C$, and that $D\subseteq \mathbf{F}$ is a definable subset  of complexity at most $M$. Then the bipartite graph $((\mathbf{F},+),(\mathbf{F},+),x-y\in D)$ is $C|\mathbf{F}|^{-1/2}$-quasirandom. So if $\mathbf{F}$ has characteristic $>C$, then every definable subset of $\mathbf{F}$ of complexity at most $M$ is a $C|\mathbf{F}|^{-1/8}$-quasirandom subset of $(\mathbf{F},+)$.
 \end{proposition}

 Let us make a brief comparison with some of the results from \cite{green}. Recall from Fact \ref{green_theorem} the `finite field model' case of the regularity lemma from \cite{green}: for any prime $p$ and $\varepsilon\in(0,1/2)$, there is $C=C(p,\varepsilon)>0$ such that, if $G$ is of form $(\mathbf{F}_{p^n},+)$ for some $n$, and $D\subseteq G$ is an arbitrary subset, then there is a subgroup $H\leqslant G$ of index at most $C$ and such that, for all but $\varepsilon|G|$-many $g\in G$, the graph $(g+H,H,x-y\in D)$ is $\varepsilon$-quasirandom. Theorem \ref{main_theorem} in the case of the additive group and Proposition \ref{char_0} relate in a similar way to that result as Tao's `algebraic regularity lemma' relates to Szemerédi's regularity lemma; for a sharply restricted class of subsets, namely those definable of bounded complexity in the ring language, we get an improved regularity result, with no irregular cosets and `power-saving' degree of quasirandomness in the size of the field. We also get an improvement that does not have an analogue in the usual graph-theoretic version, which is that the bound for the index of the subgroup is independent of the characteristic $p$, and, for sufficiently large characteristic, one does not need to pass to a subgroup at all – the definable subsets of fixed complexity are already quasirandom. In contrast, if one applies Green's theorem to subsets of $(\mathbf{F}_{p^n},+)$, then, the larger the characteristic $p$ is, the quantitatively worse of a bound one gets, with higher characteristic requiring larger index subgroups.

 \subsubsection{Heisenberg group}
 Recall the Heisenberg group $H_3$, which is the linear algebraic group over $\mathbb{Z}$ consisting of $3\times 3$ upper-triangular matrices with $1$s along the diagonal. This is a typical example of a unipotent subgroup of $\mathrm{GL}_3$. The $\mathbb{C}$-points $H_3(\mathbb{C})$ are simply connected as a topological group, hence in particular simply connected as an algebraic group, so $H_3$ is simply connected as an algebraic group over any field of characteristic $0$. Thus again Proposition \ref{simply connected theorem char 0} applies, telling us that, for a finite field $\mathbf{F}$ of sufficiently large characteristic, any subset of $H_3(\mathbf{F})$ definable in $\mathbf{F}$ of bounded complexity is quasirandom.

 As far as we are aware, this gives one of the first general results on quasirandom subsets of $H_3(\mathbf{F})$. As mentioned in Section \ref{quasirandom subsets of groups section}, quasirandomness is extremely well-studied for subsets of abelian groups, and appears ubiquitously in the literature on additive combinatorics; likewise, by \cite{gowers_groups} and other related work, quasirandomness is also well-studied for groups that are `very non-abelian' (for example, non-abelian simple). However, this leaves many important cases undealt with; in particular, for example, the solvable or even nilpotent case seems difficult to analyze using the above tools. The representation-theoretic tools in \cite{gowers_groups}, such as Fact \ref{qr_dimension}, of course do not give anything useful for solvable groups. On the other hand, it is not obvious to the authors how theorems such as those from \cite{green} could be generalized to the solvable or even nilpotent setting; the issue is that it is not obvious to us how to understand quasirandomness of subsets of a group $G$ even if one has an understanding of quasirandomness of subsets of $H$ and $G/H$, where $H\leqslant G$ is a normal subgroup.

 So our results in this paper seem to us of particular interest in the solvable or nilpotent case. Returning to $H_3(\mathbf{F})$, Proposition \ref{simply connected theorem char 0} tells us that, if $\mathbf{F}$ has sufficiently large characteristic, then, for every definable subset $D\subseteq H_3(\mathbf{F})$ of bounded complexity, the graph $(H_3(\mathbf{F}),H_3(\mathbf{F}),xy^{-1}\in D)$ is quasirandom. Thus, for example, one might take $D$ to be the set of matrices in $\mathrm{H}_3$ all of whose entries are quadratic residues, giving an analogue of the Paley graphs but in $\mathrm{H}_3$. Quasirandomness of these graphs could likely also be proven directly, as for the Paley graphs, but the point is that it follows `for free' from Proposition \ref{simply connected theorem char 0}.
 
 Let us mention some morally related work, which is the paper \cite{heisenberg_graphs} and the family of related work cited there. In those papers, the authors study bipartite Cayley graphs of the form $(H_3(\mathbf{F}),H_3(\mathbf{F}),xy^{-1}\in D)$. In contrast with our case, they study the case where the graph is `sparse': either $D$ is finite of bounded size, or it is definable of bounded complexity but still sparse. (One concrete example is the set $D$ of matrices in $H_3$ exactly two of whose upper diagonal entries are $0$; this has density $\approx1/|\mathbf{F}|$ in $H_3(\mathbf{F})$.)
 
 Sparse graphs are trivially quasirandom, and so our result does not say anything interesting for such subsets. However, the work of \cite{heisenberg_graphs} studies what could be considered an appropriate analogue of quasirandomness for sparse graphs, which is the spectral gap of the adjacency matrix of $(\mathrm{H}_3(\mathbf{F}),\mathrm{H}_3(\mathbf{F}),xy^{-1}\in D)$; see for instance Chapter 3 of \cite{zhao} for a precise statement of the connection.

 \subsection{Semisimple case}\label{semisimple case section}
By Theorem \ref{simply connected theorem}, the definable subsets of a connected, simply connected linear algebraic group are already quasirandom, without having to pass to a subgroup. 

Among the simply connected algebraic groups, the main distinguished class with a good structure theory is that of the semisimple ones. In the semisimple simply connected case, however, one can say substantially more than Theorem \ref{simply connected theorem}, and arbitrary subsets will be quasirandom, not just definable ones; this is the most general setting akin to Example \ref{sl2_example}. In this section we will briefly point out why that is the case; it is well-known to experts and is an immediate consequence of results in the literature, but we include the discussion for completeness. We would like to thank Emmanuel Breuillard for helpful discussion and for pointing us to some references; in particular, see Proposition 6.1 of his notes \cite{breuillard_2}, which gives essentially the same observation.

First let us recall the relevant definitions. Let $\mathbf{F}$ be a field, let $\mathbf{K}$ be an algebraically closed field containing $\mathbf{F}$, and let $G$ be a connected algebraic group over $\mathbf{F}$. We say $G$ is `almost-simple' if every proper normal connected algebraic (over $\mathbf{K}$) subgroup of $G$ is trivial; in this case $Z(G)$ is finite and $G(\mathbf{K})/Z(G(\mathbf{K}))$ is simple as an abstract group. For example, $\mathrm{SL}_2$ is almost simple in any characteristic. We say $G$ is `semisimple' if every proper normal connected abelian algebraic (over $\mathbf{K}$) subgroup of $G$ is trivial. If $G$ is semisimple, then there is a connected, simply connected, semisimple algebraic group $\widetilde{G}$ and a central isogeny $\widetilde{G}\to G$, all defined over $\mathbf{F}$ and unique up to isomorphism over $\mathbf{F}$, called the `universal cover' of $G$. Moreover, there are connected almost-simple algebraic subgroups $G_1,\dots,G_s$ of $\widetilde{G}$, defined over $\mathbf{F}$, such that the multiplication map $G_1\times\dots\times G_s\to\widetilde{G}$ is an isomorphism of algebraic groups (ie, it is both a morphism of algebraic groups and an isomorphism of varieties).

Now, suppose $G$ is a simply connected almost-simple algebraic group defined over a finite field $\mathbf{F}$. Note that, for sufficiently large finite fields $\mathbf{F}'\geqslant\mathbf{F}$, the abstract group $G(\mathbf{F}')$ is simple modulo its center; this is clasically known and there are many ways to see it, but in the spirit of this paper we point out a quick pseudofinite argument. Suppose there are arbitrarily large finite $\mathbf{F}'\geqslant\mathbf{F}$ with $G(\mathbf{F}')/Z(G(\mathbf{F}'))$ not simple. Let $\mathbf{F}_0$ be an ultraproduct. By Corollary 5.3 of \cite{hrushovski_pillay_2}, $G(\mathbf{F}_0)/Z(G(\mathbf{F}_0))$ is simple as an abstract group. On the other hand, $\mathbf{F}_0$ is $\aleph_0$-saturated (and in fact $\aleph_1$-saturated), so the conjugacy class of any non-trivial element of $G(\mathbf{F}_0)/Z(G(\mathbf{F}_0))$ must generate it in a bounded number of steps. So, by Łoś's theorem, the same is true for one of the $\mathbf{F}'$. This contradicts that $G(\mathbf{F}')/Z(G(\mathbf{F}'))$ is not simple.

By the main theorem of \cite{landazuri_seitz}, it then follows that there is an absolute constant $C$ such that, for all finite $\mathbf{F}'\geqslant\mathbf{F}$, every non-trivial complex projective representation of $G(\mathbf{F}')/Z(G(\mathbf{F}'))$ has dimension at least $C|\mathbf{F}'|$. It follows that every non-trivial complex linear representation of $G(\mathbf{F}')$ has dimension at least $C|\mathbf{F}'|$; indeed, if $\pi:G(\mathbf{F}')\to\mathrm{GL}(V)$ is an irreducible complex linear representation, then each element of $Z(G(\mathbf{F}'))$ acts $G(\mathbf{F}')$-linearly on $V$, and so by Schur's lemma must act as multiplication by some scalar. So $Z(G(\mathbf{F}'))$ acts trivially on the induced projective representation $G(\mathbf{F}')\to\mathrm{PGL}(V)$, and hence $\pi$ gives rise to a complex projective representation of $G(\mathbf{F}')/Z(G(\mathbf{F}'))$. On the other hand, by a similar ultraproduct argument as above, $G(\mathbf{F}')$ is perfect for sufficiently large $\mathbf{F}'$, in which case some element of it must act on $V$ as a non-scalar, whence the induced complex projective representation of $G(\mathbf{F}')/Z(G(\mathbf{F}'))$ is non-trivial. So indeed the claim follows from \cite{landazuri_seitz}. Altogether, we get the following:

\begin{fact}\label{cfsg}
    Suppose $\mathbf{F}$ is a finite field, and that $G$ is a simply connected almost-simple algebraic group over $\mathbf{F}$. Then every nontrivial complex representation of the abstract group $G(\mathbf{F})$ has dimension at least $C|\mathbf{F}|$, where $C$ is a constant depending only on the complexity of $G$.
\end{fact}

Note that, if $H$ is an abstract group generated by subgroups $H_1,\dots,H_n$, and every nontrivial representation of $H_i$ has dimension at least $m_i$, then every nontrivial representation of $H$ has dimension at least $\min_im_i$. So we get the following by Fact \ref{cfsg} and the structure theory of semisimple algebraic groups:

\begin{corollary}
    Suppose that $\mathbf{F}$ is a finite field, and that $G$ is a simply connected semisimple algebraic group over $\mathbf{F}$. Then every nontrivial complex representation of the abstract group $G(\mathbf{F})$ has dimension at least $C|\mathbf{F}|$, where $C$ is a constant depending only on the complexity of $G$.
\end{corollary}
\begin{proof}
    Since $G$ is simply connected and semisimple, there are almost-simple algebraic subgroups $G_1,\dots,G_s$ of $G$, defined over $\mathbf{F}$, such that the multiplication map $G_1\times\dots\times G_s\to G$ is an isomorphism of group varieties. It follows that the multiplication map $G_1(\mathbf{F})\times\dots\times G_s(\mathbf{F})\to G(\mathbf{F})$ is an isomorphism of abstract groups. Let $r_i$ be the rank of $G_i$ and let $d_i$ be the Zariski dimension of $G_i$. By Fact \ref{cfsg}, any nontrivial complex representation of $G_i(\mathbf{F})$ has dimension at least $|\mathbf{F}|^{r_i}/C_i$, where $C_i$ depends only on $d_i$. So every nontrivial complex representation of $G(\mathbf{F})$ has dimension at least $\min_i(|\mathbf{F}|^{r_i}/C_i)\geqslant|\mathbf{F}|/\max_i(C_i)$.
\end{proof}

Now by Fact \ref{qr_dimension} one immediately gets:

\begin{corollary}\label{simply connected semisimple is quasirandom}
    Suppose that $\mathbf{F}$ is a finite field, and that $G$ is a simply connected semisimple algebraic group over $\mathbf{F}$. Then every subset of the abstract group $G(\mathbf{F})$ is $C|\mathbf{F}|^{-1/2}$-quasirandom, where $C$ is a constant depending only on the complexity of $G$.
\end{corollary}

Suppose now that $G$ is semisimple but not simply connected. We can still get quasirandomness of the $\mathbf{F}$-points, but we need to pass to an appropriate bounded-index subgroup. This is because of the following, which was proved in \cite{hrushovski_pillay_2} and independently in \cite{sclosa}:

\begin{fact}
    Suppose $\mathbf{F}$ is a finite field, $G$ is a semisimple algebraic group defined over $\mathbf{F}$, and $f:G'\to G$ is an isogeny defined over $\mathbf{F}$. If the kernel of $f$ has size $n$, then $f(G'(\mathbf{F}))$ is a subgroup of $G(\mathbf{F})$ of index $n$.
\end{fact}

In particular, if $G$ is semisimple but not simply connected, then $G(\mathbf{F})$ will have proper subgroups of small index, and hence cannot be quasirandom. However, we still get the following:

\begin{proposition}
    Suppose that $\mathbf{F}$ is a finite field and that $G$ is a connected semisimple algebraic group over $\mathbf{F}$. Let $f:\tilde{G}\to G$ be its universal cover, and suppose its kernal has size $n$. Then $f(\tilde{G}(\mathbf{F}))$ is a subgroup of $G(\mathbf{F})$ of index $n$, every subset of which is $C|\mathbf{F}|^{-1/2}$-quasirandom, where $C$ is a constant depending only on the complexity of $G$.
\end{proposition}

}

\end{document}